\documentclass[11pt,reqno]{amsart}
\usepackage{amsmath}
\usepackage{amsfonts}
\usepackage{mathrsfs}
\usepackage{amssymb}
\usepackage{xypic}
\usepackage{amsthm}
\usepackage{url}
\usepackage{setspace}     \spacing{1}
\usepackage{enumitem}
\usepackage{comment}
\usepackage{bbm}
\usepackage[colorlinks=true]{hyperref} \hypersetup{urlcolor=blue, citecolor=red}
\usepackage{AaronMacro}
\usepackage[margin=1in]{geometry}
\theoremstyle{Theorem}
\newtheorem{custheorem}{Theorem} 
\usepackage{pstricks, pst-plot, pst-node, pst-math,pst-3dplot,pstricks-add}
\newcommand{\diff}{\mathrm{Diff}}
\newcommand{\Diff}{\diff}

\newcommand{\Cmu}[1][\mu]{{\diff^{{r}}}(\T^2; #1)} 

\newcommand{\Gl}{\mathrm{GL}}
\newcommand{\GL}{\mathrm{GL}}

\title[Rigidity of measures on the torus]{Rigidity of measures on the torus: smooth stabilizers and entropy}
\author{Aaron W. Brown}
\long\def\symbolfootnote[#1]#2{\begingroup\def\thefootnote{\fnsymbol{footnote}}
\footnote[#1]{#2}\endgroup}

\def\I{\mathcal I}
\def\Exp{\mathbb E}

\begin{document}
%\symbolfootnote[0]{\it Preliminary version.  Last updated: \today}
\maketitle

\begin{abstract}
For a nonlinear Anosov diffeomorphism $a$ of the 2-torus, we  present examples of measures so that the 
group of $\mu$-preserving diffeomorphisms is, up to zero-entropy transformations, cyclic.  For families of 
equilibrium states $\mu$, we strengthen this to show that the group of $\mu$-preserving diffeomorphism is virtually cyclic.  
\end{abstract}

\section{Introduction and statement of results}
A common problem in the theory of dynamical systems is the following:
\begin{quote}\it Given  a group action on a metric space, find, classify, and study the properties of invariant Borel probability measures.
\end{quote}
In the case that $X$ is compact and the group is generated by a homeomorphism $f\colon X\to X$, the Krylov-Bogolyubov theorem guarantees the existence of at least one invariant Borel probability measure.  
For $\Z$- or $\R$-actions on manifolds generated by Axiom A diffeomorphisms or flows, there are many mutually singular, invariant Borel probability measures.  Indeed, for every cohomology class of a H\"older continuous function $\phi \colon X\to \R$ there exists a distinct invariant Borel probability measure, called the equilibrium state for $\phi$.  (See \cite{MR2423393} and below for background in equilibrium states.)

In contrast, there are many well studied situations in which there are relatively few invariant measures.  
For instance, Rudolph showed  in \cite{MR1062766} that the only Borel measure on $S^1$ ergodic under the semigroup action generated by $z\mapsto z^2$ and $z\mapsto z^3$, and such that each 1-parameter sub-action has positive entropy, is the Lebesgue measure.  Generalizations of Rudolph's result to algebraic actions by higher rank abelian groups and semi-groups have been studied, for instance, in  \cite{KK} %, \cite{MR2261075}, 
and \cite{MR1406432}.

In this article we consider a related but far less studied problem:
\begin{quote}\it Given a manifold equipped with a Borel probability measure $\mu$, classify---or find non-trivial constraints on---the group of $\mu$-preserving diffeomorphisms.    \end{quote}
For two extreme cases the group of $\mu$-preserving diffeomorphisms is in some sense too large to admit interesting constraints.  On one extreme, if $\mu$ is a volume the group of $\mu$-preserving diffeomorphisms is an infinite dimensional manifold.   On the other extreme, if $\mu$ is a supported on a finite number of points, the group of $\mu$-preserving  diffeomorphisms is a  finite-codimensional manifold in the space of all diffeomorphisms of $M$.  
Thus a natural class of measures in which to first consider the above problem is the class of singular measures with full support.  In the present article we restrict ourselves to  large families of such measures on the 2-torus.  % and study the above problem for large classes of singular measures with full support.  

\subsection{Notation and definitions}
Consider a compact $C^\infty$ manifold  $M$ and $\{\mu_i\}$ a collection of Borel probability measure on $M$.  For $r \ge 1$ we write  $\mathrm{Diff}^r(M; \{\mu_i\})$ for the group of  $C^r$ diffeomorphisms $f\colon M \to M$ such that  $f_*\mu_i = \mu_i$ for all $i$.   For $\{\Fol_i\}$ a family of foliations on $M$ we write  $\diff^r(M; \{\Fol_i\})$ for the group of $C^r$ diffeomorphisms preserving each foliation $\Fol_i$.

We identify the torus $\T^k$ in the standard way with the quotient $\R^k/\Z^k$.  For $A\in \GL(k, \Z)$ we write $L_A\colon \T^k\to \T^k$ for the induced toral automorphism.  

Recall that a diffeomorphism of a compact manifold $f\colon M\to M$ is called \emph{Anosov} if, for any Riemannian metric, there are constants $C>0$,  ${0<\kappa<1}$ and a continuous $Df$-invariant splitting of the tangent bundle $T_x M = E^s(x)\oplus E^u(x)$ 
so that for every  $x\in M$ and $n \in \N$
\begin{align*}
 \|Df^n_x v\|\le C  \kappa^n \|v\|,  \quad &\mathrm{ for } \ v\in E^s(x)\\
\|Df^{-n}_x v\|\le C\kappa^{n}\|v\|,  \quad &\mathrm{ for }\   v\in E^u(x).
\end{align*}
It is well known that for any Anosov diffeomorphism $f$ of the torus $\T^k$, there exists a hyperbolic linear map $A\in \GL(k, \Z)$ and $h\colon \T^k \to \T^k$ with 
\[h\circ f \circ h\inv = L_A\]
where $h$ and $h\inv$ are H\"older continuous.  A similar result holds if we replace $\T^k$ with any compact nil-manifold, where $A$ is replaced with a corresponding nil-automorphism  (see \cite{Manning:1973p11377}).

For ${f\colon X\to X}$ a homeomorphism of  a compact metric space  and ${\phi\colon X \to \R}$ a continuous function, we say an $f$-invariant measure $\mu$ is an \emph{equilibrium state for $\phi$} (with respect to $f$) if $\mu$ maximizes the  expression
\[h_\mu(f) + \int \phi\ d \mu\]
over all $f$-invariant measures.  Here $h_\mu(f)$ denotes the measure theoretic entropy of $f$ with respect to $\mu$.  (See, for example, \cite{MR1326374} for the definition and properties of $h_\mu(f)$.)  For $f$ Anosov and $\phi$ H\"older continuous, there exists a unique equilibrium state $\mu_\phi$.  It is well known that the equilibrium state for a H\"older continuous function with respect to an Anosov diffeomorphism (or more generally,  a basic set for an Axiom A diffeomorphism)  is ergodic, has full support, and has positive entropy.  In addition the equilibrium state  $\mu_\phi$  possesses a \emph{local product structure}, which we will describe formally in  Theorem \ref{thm:prodstruct} below.   We refer to \cite{MR2423393} for background on equilibrium states in the uniformly hyperbolic setting.
 
For an Anosov diffeomorphism $f$, there are three `natural' equilibrium states:
\begin{itemize}
\item the \emph{forwards SRB} measure, the equilibrium state for $$\phi^u:= -\log \left(\det \big(\restrict {D f}{E^u} \big)\right);$$
\item the \emph{backwards SRB} measure, the equilibrium state for $$\phi^s:= -\log \left(\det \big(\restrict {D f\inv}{E^s} \big)\right);$$
\item the \emph{measure of maximal entropy}, the equilibrium state for $\phi\equiv 0$.
\end{itemize}
We note that in the case that $f$ is algebraic all three measures coincide.  When $f$ is volume preserving, the forwards and backwards SRB measures coincide.

\subsection{Statement of results}
To state the results, fix $\theta\in (1, \infty]$ and a (non-linear) $C^{\theta}$ Anosov diffeomorphism \[a\colon \T^2 \to \T^2.\]  (Recall that a diffeomorphism is said to be of class $C^{k+\alpha}$ for $k\in \N$ and $\alpha\in (0,1)$ if its derivatives of order $k$ exist and are H\"older continuous with exponent at least $\alpha$.)   %Associated to $a$ we have a hyperbolic matrix $A\in \GL(2,\Z)$ and a bi-H\"older homeomorphism $h\colon \T^2\to \T^2$ with \[ L_A \circ h = h \circ f. \]
For $\sigma \in \{s,u\}$ and  $v\in E^\sigma(x) \sm \{0\}$ define the functions 
\begin{align}  \lambda^\sigma (x) := \lim_{n\to \pm \infty} \frac{1}{n}\log( \|Da_x^n v\|).  \label{eq:lyp} \end{align}
By Oseledec's Theorem \cite{MR0240280} there is a set $\Lambda\subset \T^2$,  with $\mu(\Lambda)=1$ for any $a$-invariant Borel probability measure $\mu$, so that  for every $x\in \Lambda$ the limits in \eqref{eq:lyp} exist.  
In the case when $\mu$ is $a$-ergodic the functions $\lambda^u(\cdot )$ and $\lambda^s(\cdot )$ are constant $\mu$-a.e.\ whence we  write $\lambda_\mu^u$ and $\lambda_\mu^s$ for these constants.

\begin{thm}\label{thm:2}
Let $\mu$ be an $a$-ergodic measure on  $\T^2$ with $h_\mu(a)>0$ and full support.  If $\mu$  satisfies \[\lambda_\mu^u \neq - \lambda_\mu^s\] then for $r>1$, %$r\ge 1+\alpha$, 
then
\begin{enumerate}
\item the set of zero-entropy diffeomorphisms $N= \{ g\in \Cmu \mid h_\mu(g) = 0\}$ is a normal subgroup of $\Cmu$;
\item there is a natural isomorphism $\Cmu/N \cong \Z$  whenever $\Cmu\neq N$.
\end{enumerate}\end{thm}

\noindent In particular, the entropy of $\mu$ is \emph{quantized} in the sense that 
	\[\inf\{h_\mu (g)\mid g\in \mathrm{Diff}^r(M; \mu) \text{ and } h_\mu(g)>0\} >0.\]  Note that the group $\Cmu\neq N$ whenever $r\le \theta$ as $a\in \Cmu$.  

%$\diff^r$-entropy-spectrum of $\mu$ is of the form  $\lambda\N$ for some $\lambda\in [0, \infty)$.   

%In particular if there is some $g\in \Cmu$ with $h_\mu(g)  >0$ then $\mu$ is  $\diff^r$-quantized.
%Note that for $r<\theta$, we are guaranteed  $\lambda \neq 0$ in the conclusion of Theorem \ref{thm:2}. 
 
For a smaller class of measures, we are able to give a more precise description of the group $\Cmu$.   

\begin{thm}\label{thm:1}
Let  $\mu$ be an equilibrium state for a H\"older continuous potential (with respect to $a$) that is neither the measure of maximal entropy, nor the forwards or backwards SRB measures. Assume in addition that \[\lambda_\mu^u +\lambda_\mu^s \neq 0.\]  Then for any $r\ge 1$ there is an $m\in \N$ so that the cyclic subgroup generated by $a^m\colon \T^2\to \T^2$ %$\{a^n \mid n\in \Z\}$
 has finite index in $\Cmu$.  In particular, $\Cmu$ is either finite or virtually infinite cyclic.  
 
% \cmt{add not maximal entropy}
\end{thm}
Recall that a group $G$ is called \emph{virtually infinite cyclic} if there is a finite index subgroup $G'\subset G$ with $G' \cong \Z$. 
We note that for $r\le \theta$, we can take $m = 1$ in the conclusion of Theorem \ref{thm:1}. Note however that we do not rule out the possibility that $m = 0$; that is, in the case that there are no infinite order elements in $\Cmu$ when $r>\theta$.  

Using similar arguments we obtain the following.  
\begin{custheorem}\label{thm:1'}
Let  $\mu, \nu$ be two $a$-invariant, ergodic Borel probability measures with  full support.  Assume  $h_\mu(a)>0$,  $h_\nu(a)>0$, and \[\lambda_\nu^u +\lambda_\nu^s<0< \lambda_\mu^u +\lambda_\mu^s.\]   Then for any $r\ge 1$ there is an $m\in \N$ so that the cyclic subgroup  generated by $a^m\colon \T^2 \to \T^2$  has finite index in $\Cmu[\{\nu, \mu\}]$.  In particular, $\Cmu[\{\nu, \mu\}]$ is either finite or virtually infinite cyclic.  
\end{custheorem}

\begin{comment}
The techniques we employ to establish Theorem \ref{thm:1} fail for measures that are not equilibrium states and for the two SRB measures for $a$.  However, for a larger class of measures, including many non-equilibrium states and the two SRB measures, we show that, in terms of entropy, the group $\Cmu$ behaves essentially like a cyclic group.  % for a larger class of measures than satisfy the hypothesis of Theorem \ref{thm:1}.  
For a measure $\mu$ on a manifold $M$, we define the \emph{$\Diff^r$-entropy-spectrum} of $\mu$ to be the set
	\[\{h_\mu (g)\mid g\in \mathrm{Diff}^r(M; \mu) \}.\]
We say that a measure $\mu$ is \emph{$\Diff^r$-quantized} if
	\[\inf\{h_\mu (g)\mid g\in \mathrm{Diff}^r(M; \mu) \text{ and } h_\mu(g)>0\} >0.\]  
	
\end{comment}
We emphasize  that Theorem \ref{thm:1}  holds for $r=1$, where as Theorem \ref{thm:2} requires the additional hypothesis that $r>1$.  
The hypothesis in all our theorems that  $\lambda_\mu^u \neq - \lambda_\mu^s$ forces the dynamics $a\colon \T^2\to \T^2$ to be non-linear and the measure $\mu$ to be singular with respect to the Riemannian volume.

\section{Preliminaries}

We begin with some basic properties of Anosov diffeomorphisms followed by background in the theory of non-uniform hyperbolicity.  For more background and proofs, we refer the reader to \cite{MR1326374} in the Anosov and uniformly hyperbolic setting, and to \cite{MR2348606} in the non-uniform setting.

\subsection{Anosov diffeomorphisms of $\T^2$}
It is well known that the only surface supporting an Anosov diffeomorphism is the torus $\T^2$.  Fix $f\colon \T^2\to \T^2$ a $C^{1+\alpha}$ ($\alpha\in (0,1)$) Anosov diffeomorphism.  Then for every $x\in \T^2$ there are $C^{1+\alpha}$ injectively immersed curves $\stab x $ and $\unst x$, called the \emph{stable} and \emph{unstable manifolds}, satisfying 
\begin{align*}
\stab x &:= \{y\in \T^2 \mid d(f^n(x), f^n(y)) \to 0  \text{  as  } n\to \infty\};\\
\unst x &:= \{y\in \T^2 \mid d(f^{-n}(x), f^{-n}(y)) \to 0  \text{  as  } n\to \infty\}.
\end{align*}
For sufficiently small $\epsilon$ we also have that the sets 
\begin{align*}
\locStab x &:= \{y\in \T^2 \mid d(f^n(x), f^n(y)) \le \epsilon \text{  as  } n\to \infty\};\\
\locUnst x &:= \{y\in \T^2 \mid d(f^{-n}(x), f^{-n}(y)) \le \epsilon \text{  as  } n\to \infty\},
\end{align*}
called the  \emph{local} stable and {unstable manifolds}, are $C^{1+\alpha}$-embedded curves.

Recall that a \emph{$d$-dimensional} \emph{$C^{r,k}$ foliation} $\Fol$ of an $n$-manifold $M$ is a partition of $M$ by immersed submanifolds $\{\Fol(x)\}_{x\in M}$, and a cover by open sets $\{U_\beta\}$ such that 
\begin{enumerate}
\item the connected component of $\Fol(x) \cap U_\beta$ containing  $x$, which we denote by $\Fol_{U_\beta}(x)$, is a $C^r$-embedded copy of $\R^d$ for every $x\in U_\beta$ and all $\beta$;
\item there are coordinate maps  \[\phi_\beta\colon \R^{d}\times \R^{n-d} \to U_\beta\] such that \[\phi_\beta\colon \R^{d}\times \{y\} = \Fol_{U_\beta}(\phi_\beta(0,y));\]
\item on the intersection $U_\beta\cap U_\alpha$ the transition maps \[\phi_\beta\inv \circ \phi_\alpha\colon  \phi_\alpha \inv(U_\alpha)\subset  \R^n\to \R^n\] are $C^k$.
\end{enumerate}

 Here $\Fol(x)$ is called the \emph{leaf through $x$},   $\Fol_{ U_\beta}(x)$ is  called the \emph{local leaf through $x$}, and $U_\beta$ is called a \emph{foliation chart}.  In general, given a foliation $\Fol$ of $M$ and an open set $V\subset M$ we denote by $\Fol_V$ the \emph{local foliation} of $V$ whose leaf through $x$ is the connected component of $\Fol(x) \cap V$ containing $x$.

For $f$ an Anosov diffeomorphism, the partitions of $\T^2$ by stable and unstable manifolds induce foliations $\Fols$ and $\Folu$.  When working with the foliations $\Fols$ and $\Folu$ we write $W^\sigma_V(x)$ for the  leaf of $\Fol_V^\sigma$ through $x$.  By a \emph{$C^r$ bifoliation chart} for the foliations $\Fols$ and $\Folu$ we mean an open set  $V\subset \T^2$ and a $C^r$ diffeomorphism \[\phi\colon \R \times \R \to V\] with 
\[\phi\colon \{x\} \times \R\mapsto \locStab[V] {\phi(x,0)}\quad \text{and} \quad \phi\colon  \R\times \{y\} \mapsto \locUnst[V] {\phi(0,y)}.\]

In higher dimensions one needs to be careful about the regularity of the foliations $\Folu$ and $\Fols$: typically each foliation is at best $C^{1+\alpha, \text{H\"older}}$.  However, in our setting the low ambient dimension  guarantees stronger regularity.  The following is well known.  (See, for example, \cite{MR1887706} and \cite[Theorem 6.1]{MR1432307}; note that we need the hypothesis that the dynamics is  at least $C^{1+\alpha}$.)
% and can be deduced, for example, from \cite{MR1304137}.  %(Note 

\begin{prop}\label{prop:folreg}
Let $f\colon \T^2 \to \T^2$ be a $C^{1+\alpha}$ Anosov diffeomorphism.  Then the unstable and stable foliations $\Folu$ and $\Fols$ are  $C^{1+\alpha, 1+\alpha'}$ for some $\alpha'$.
\end{prop}
For $U$ a  foliation chart for $\Folu$ and embedded curves $D, D' \subset U$ with $D$ and $D'$ transverse to each of the local leaves $\{\locUnst[U]x\} _{x\in U}$, we define the \emph{unstable holonomy} maps
\[h^u_{D, D'} \colon G\subset D \to D'\] 
by \[h^u_{D, D'}\colon z \mapsto D' \cap \locUnst[U]z\]
when defined.  As a consequence of Proposition \ref{prop:folreg} we obtain that the unstable holonomy maps $h_{D, D'}$ are $C^{1+\alpha'}$; in particular they are bi-Lipschitz.  \emph{Stable holonomy} maps are defined similarly and are also bi-Lipschitz.

\subsection{Lyapunov exponents}\label{sec:lyap}
Let $f\colon M \to M$ be a $C^{1+\alpha}$ diffeomorphism of a Riemannian manifold.  
We recall that there is a Borel subset $\Lambda\subset M$, called the set of \emph{regular points}, Borel functions $r(x)$, and %\colon \Lambda\to \N$ and 
$$ \lambda_0(x) < \lambda_1(x) < \dots < \lambda_{r(x)}(x)$$ on $\Lambda$, and  a decomposition of the tangent space \[T_xM = \bigoplus_{0\le j \le r(x) } E^j(x)\]  over $ \Lambda$ so that  (among other properties) for $x\in \Lambda$ and  $v\in E^j(x)\sm\{0\}$ \[\lambda_j(x):=\lim_{n\to \pm \infty} \dfrac{1}{n} \log\left(  \|Df^n_x(v)\| \right). \] 
For $x\in \Lambda$, the numbers $\{\lambda_i(x)\}$ are called the \emph{Lyapunov exponents} at $x$ and the subspaces $E^j(x)$ are called the \emph{Lyapunov subspaces} at $x$.  
By Oseledec's Theorem \cite{MR0240280}, for any $f$-invariant Borel probability measure $\mu$ we have that $\mu(\Lambda)=1$, and the splitting $T_xM = \oplus_{0\le j \le r(x) } E^j(x)$ depends $\mu$-measurably on the point $x$.

By \cite{MR0458490},  for every $x\in \Lambda$ and  $0\le i\le r(x)$ with $\lambda_i(x)> 0$ there is a $C^{1+\alpha}$ injectively immersed  $\left({\sum_{\lambda_j(x)\ge \lambda_i(x)} \dim E^j(x)}\right)$-dimensional open manifold $\wtd W^i(x)$ defined by 
\[\wtd W^i(x) := \left\{ y\in M \mid \liminf_{n\to \infty}-\dfrac{1}{n} \log(d(f^{-n}(y), f^{-n}(x)))\ge \lambda_i(x)\right\} \]
with \[T_x\wtd W^i(x) = \bigoplus_{\lambda_j(x)\ge \lambda_i(x)} E^j(x)\] called the \emph {$i^{\text {th}}$ unstable Pesin manifold}  at $x$.  
Similarly defined \emph{stable Pesin manifolds} $\wtd W^i(x)$ exist for $x\in \Lambda$ with  $\lambda_i(x)< 0$.

When $M= \T^2$ and $f$ is Anosov,  for any regular point $x$ we have $r(x) = 1$ and  $\lambda_0(x) <0<\lambda_1(x)$.   In this context and write $\lambda^s = \lambda_0$ and $\lambda^u = \lambda_1$ as in \eqref{eq:lyp} for the \emph{stable} and \emph{unstable Lyapunov exponents}.  Clearly  in this context, for any regular point $x\in \T^2$ we have  $\wtd W^1(x) \subset \unst x$ and $\wtd W^0(x) \subset \stab x$.

\subsection{Conditional measures}
Recall (see, for example, \cite{MR0047744}) that given a measurable partition $\xi$ of  Lebesgue space $(X, \mu)$ one may find  a collection of measures $\{\td \mu^\xi_x\}_{x\in X}$, called a \emph{family of conditional probability measures}, such that \begin{enumerate}
\item $\td \mu^\xi_x = \td \mu^\xi_y$ for $y\in \xi(x)$;
\item $\td \mu^\xi_x(\xi(x)) = 1$ and $\td \mu^\xi_x(X\sm \xi(x)) = 0$ for $\mu$-a.e.\ $x$;  
\item for measurable subsets $A\subset X$ the functions $x\mapsto \td  \mu^\xi_x(A)$ are measurable and 
\[\mu(A) = \int _X \td  \mu^\xi_x(A) \ d \mu(x);\]
\item any other collection of measures satisfying  (1)-(3) is equivalent to $\{\td \mu^\xi_x\}_{x\in X}$  on a set of full measure.
\end{enumerate}
We need the following straightforward observation.
\begin{claim}\label{clm:condtransfer} Let $(X, \mu)$ be a Lebesgue space, $\xi$  a measurable partition, and $g\colon X\to X$ an invertible measure preserving transformation.  Let $\{\td \mu^\xi_x\}$ and $\{\td \mu^{g(\xi)}_x\}$ be families of conditional probability measures for the partitions $\xi$ and $g(\xi)$.  Then for $\mu$-a.e.\ $x$ \[g_*(\td \mu^\xi_x) = \td \mu^{g(\xi)}_{g(x)}.\]\end{claim}

\subsection{Pointwise dimension of measures}
For $X$ a metric space, and $\mu$ a locally finite Borel measure, we define the \emph{upper} and \emph{lower pointwise dimension} functions
\begin{align*}
\overline \dim (\mu, x) &:= \limsup_{\epsilon\to 0} \dfrac{\log \mu(B(x, \epsilon))}{\log \epsilon}\\
\underline \dim (\mu, x)& := \liminf_{\epsilon\to 0} \dfrac{\log \mu(B(x, \epsilon))}{\log \epsilon}
\end{align*}
where $B(x, \epsilon)$ denotes the metric ball of radius $\epsilon$ at $x$
and the \emph{pointwise dimension} function
\[\dim (\mu, x) := \lim_{\epsilon\to 0} \dfrac{\log \mu(B(x, \epsilon))}{\log \epsilon}\]
whenever the limit is defined.

For a $C^{1+\alpha}$ diffeomorphism $f\colon M\to M$ and an $f$-ergodic Borel probability measure $\mu$ on $M$, the functions $r, \lambda_i,$ and $\dim E^i$ are a.e.\ constant.  For an ergodic $\mu$ and an $i$ with $\lambda_i>0$, the collection $\{\wtd W^i(x)\}_{x\in \Lambda}$ (and its measure zero complement) provides a partition of $M$.  We say that a measurable partition $\xi$ is \emph{subordinate to $\{\wtd W^i(x)\}_{x\in \Lambda}$} if for $\mu$-a.e.\ $x$ we have $\xi(x)\subset\wtd  W^i(x)$ and $\xi(x) \cap  \wtd W^i(x)$ contains an open neighborhood of $x$ in $\wtd W^i(x)$.  Let $\xi$ be a measurable partition subordinate to $\{\wtd W^i(x)\}_{x\in \Lambda}$ and let $\{\td \mu^\xi_x\}_{x\in M}$ be a family of conditional probability measures.   We define measurable functions
\begin{align*}
\overline \delta^i(x) &:= \overline \dim (\td \mu^\xi_x,x) := \limsup_{\epsilon\to 0} \dfrac{\log \td \mu^\xi_x(B(x, \epsilon))}{\log \epsilon}\\
\underline \delta^i(x) &:= \underline \dim (\td \mu^\xi_x,x) := \liminf_{\epsilon\to 0} \dfrac{\log \td \mu^\xi_x(B(x, \epsilon))}{\log \epsilon}.
\end{align*}
From  \cite{MR819557} we have the equality
\[\overline \delta^i ( x)= \underline \delta^i ( x)\]
 at $\mu$-almost every $x$; we define $\delta^i ( x)$ to be this common value. 
 %If we replace $\xi$ with another $\{\wtd W^i(x)\}_{x\in \Lambda}$-subordinate partition, the functions $\overline \delta^i(x), \underline \delta^i(x),$ and $  \delta^i(x)$ remain unchanged for $\mu$-a.e.\ $x$.  

Since  % of the functions $\dim (\mu, x)$ and 
$\delta^i (\cdot)$ is measurable, the assumption that $\mu$ is ergodic guarantees it is  $\mu$-a.e.\ constant.
In the case that $\mu$ is not ergodic, the functions $\overline \delta^i(x), \underline \delta^i(x)$, and $\delta^i(x)$ are still defined $\mu$-a.e.\ by first passing to an ergodic decomposition (see \cite{MR819557} for details).  

We also define measurable functions for the \emph{stable} and \emph{unstable} dimension of the measure $\mu$:
\[\delta^u (x) = \max\{\delta^i(x)\mid \lambda_i(x) >0\};\]
\[\delta^s (x) = \max\{\delta^i(x)\mid \lambda_i(x) <0\}.\]
A measure $\mu$ is said to be \emph{hyperbolic for $f$} if $\lambda_i(x) \neq 0$ for $\mu$-a.e regular point $x$ and every $0\le i \le r(x)$.  From \cite{MR1709302}, we have that for an ergodic, hyperbolic measure \begin{align}\label{eq:exdim}\dim(\mu) = \delta^u + \delta^s\end{align}
where $\dim(\mu),  \delta^u , \delta^s$ are the constant values attained $\mu$-a.e.\ by the corresponding functions.  
%Note, implicit in \eqref{eq:exdim} is the fact that for $\mu$-a.e.\ $x$ \begin{align*}\overline \dim (\mu, x)= \underline \dim (\mu, x)= \dim (\mu, x).\end{align*}

For $x\in \Lambda$ we write $u(x):= \inf \{0\le j \le r(x) \mid \lambda_j(x) >0\}$.  We say a measure $\mu$ is a \emph{$u$-measure} if 
for any $\{\wtd W^{u(x) }(x)\}$-subordinate measurable partition $\xi$, and corresponding family of conditional probability measures $\{\td \mu ^\xi_x\}$, for $\mu$-a.e.\ $x$ the measure $\td \mu ^\xi_x$ is absolutely continuous with respect to the induced Riemannian volume on $\wtd W^{u(x)}(x)$.  This is equivalent to the property that for $\mu$-a.e.\ $x$
\[\delta^u(x) = \sum_{j\ge u(x)} \dim E^j(x).\]
%which in turn is equivalent to the property that for $\mu$-a.e.\ $x$
%and $j\ge u(x)$
%\[ \dim E^j(x) = \begin{cases} \delta^{j}(x) -\delta^{j+1}(x) &j< r(x)\\[.5em] \delta^j(x)& j =r(x)\end{cases}.\]
We similarly define  \emph{$s$-measures}. 
We note that for $f\colon M\to M$ an Anosov diffeomorphism, the forwards (resp. backwards) SRB measure is the unique $u$- (resp. $s$-) measure for $f$.

We will rely on the following  well known technical result relating the pointwise dimension of a measure and the Hausdorff dimension of a set.  
\begin{prop}[Proposition 2.1 of \cite{MR684248}]
\label{prop:Haus}
Let $\mu$ be a non-atomic, locally finite Borel measure on a manifold and let $\mu(\Lambda)>0$.  Suppose  \[\underline \delta\le \liminf _{r\to 0}\dfrac{\log(\mu (B(x, r)))}{\log r}\le \limsup _{r\to 0}\dfrac{\log(\mu (B(x, r)))}{\log r}\le \overline \delta\] for \emph{every} $x\in \Lambda$.  Then $$\underline\delta \le \dim_H(\Lambda)\le \overline \delta$$ where $\dim_H$ denotes Hausdorff dimension.
\end{prop}

%\subsubsection{Pointwise dimension of measures under transformations}
We consider the behavior of the pointwise dimension of  a measures under a bi-Lipschitz map.  %Let $X$ be a metric space, and
Let $\nu$ and $\mu$ be two locally finite Borel measures on $\R^m$ with $\nu \ll \mu$.  Recall that for a measurable set $A\subset \R^m$, a point $y$ is said to be a $\mu$-\emph{density point of $A$} if \[\lim _{r\to 0}\frac{\mu(B(y,r) \cap A)}{\mu(B(y, r))} = 1.\]
 We say that $y $ is a \emph{bounded $(\nu,\mu)$-density point} if there is some $N\in (0,\infty)$ so that $y$ is both a $\mu$- and $\nu$-density point of the set
\[\left\{x\in \R^m \mid \dfrac{1}{N} \le \dfrac{d \nu }{d \mu}(x)\le N\right\}.\]
We note that $\nu\ll \mu$ implies $\nu$-a.e.\ point is a bounded $(\nu,\mu)$-density point.

\begin{prop}\label{prop:bddRND}
%Let $(X, \mu)$ and $(Y,\nu)$ be metric spaces equipped with locally finite Borel measures.  Let $g\colon U\subset X \to Y$ be a bi-Lipschitz homeomorphism onto its image with $g_*(\mu) \ll \nu$.  Then for each bounded $(g_*(\mu),\nu)$-density point $y \in g(U)$ we have 
Let $\mu$ and $\nu$ be locally finite Borel measures on $\R^m$. Let $g\colon  \R^m \to \R^m$ be a bi-Lipschitz homeomorphism with $g_*(\mu) \ll \nu$.  Then for each bounded $(g_*(\mu),\nu)$-density point $y$ we have 
\begin{enumerate}
%\item $\dim(\nu; y)$ is defined 
\item $\overline{\dim}(\nu,y) = \overline{\dim}(\mu, g\inv(y))$;
\item $\underline{\dim}(\nu,y) = \underline{\dim}(\mu, g\inv(y)).$
\end{enumerate}
\end{prop}

\begin{proof}

We write $J(y)$ for the  Radon-Nikodym Derivative 
\(J(y) := \dfrac{d g_*\mu}{d\nu}(y).\)
%We have that $J$ is defined $g_*\mu$-a.e., hence $\nu$-a.e..  
For $N\in \N$, we write $V_N := \left\{y\mid \dfrac{1}{N}\le J(y)\le N\right\}$.
Consider the inequality
\begin{align*}
\dfrac{g_* \mu (B (y,r) )}{\nu (B(y,r))} &=    \dfrac{\int_{B(y,r)} J(z)\ d \nu (z)} {\nu (B(y,r))}\\
&\ge   \dfrac{1}{N} \dfrac{\nu (B(y,r)\cap V_N)} {\nu (B(y,r))}
\end{align*}
Since $y$ is a $\nu$-density point of $V_N$ for some $N$, we  have that $\frac{g_* \mu (B(y,r))}{\nu (B(y,r))} $ is bounded away from $0$ as $r\to 0$.

Similarly we have
\begin{align*}
\dfrac{g_* \mu (B(y,r))}{\nu (B(y,r))} 
&= N \dfrac {g_* \mu (B(y,r))}  { \int_{B(y,r)} N \ d \nu (z)}   \\
&\le N \dfrac {g_* \mu (B(y,r))}  { \int_{B(y,r)\cap V_n} J(z) \ d \nu (z)}   \\
&= N \dfrac {g_* \mu (B(y,r))}  {g_* \mu (B(y,r)\cap V_N)}
%{\dfrac{\nu (B(y,r))} {\int_{(B(y,r)} I_{V_c} (z)\ d g_* \mu (z)}}
\end{align*}
which implies $\frac{g_* \mu (B(y,r))}{\nu (B(y,r))} $ is bounded away from $\infty$ as $r\to 0$ since $y$ is a $(g_*\mu)$-density point of $V_N$ for some some $N$.  

In particular  the expression \[\log \left(\dfrac{g_* \mu (B(y,r))}{\nu (B(y,r))}\right)\] is bounded above and below for all sufficiently small $r> 0$, hence \[ \limsup_{r\to 0}\dfrac{\log \left(\frac{g_* \mu (B(y,r))}{\nu (B(y,r))}\right)}{\log r} =  \liminf_{r\to 0}\dfrac{\log \left(\frac{g_* \mu (B(y,r))}{\nu (B(y,r))}\right)}{\log r} = 0.\]

We thus have 
{ \allowdisplaybreaks[4]
\begin{align*}
\overline{\dim}(\nu,y)&:= \limsup_{r\to 0}\dfrac{\log(\nu(B(y,r)))}{\log r}\\ 
&= \limsup_{r\to 0}\dfrac{\log(\nu(B(y,r)))}{\log r}+  \limsup_{r\to 0}\dfrac{\log \left(\frac{g_* \mu (B(y,r))}{\nu (B(y,r))}\right)}{\log r}\\
&= \limsup_{r\to 0}\dfrac{\log(g_* \mu(B(y,r)))}{\log r}
\end{align*}}
and similarly
\[\underline{\dim}(\nu,y) =  \liminf_{r\to 0}\dfrac{\log(g_* \mu(B(y,r)))}{\log r}.\]

Since $g$ is assumed bi-Lipschitz, for each $y$ we may find $L, C>0$ so that $d(y,z)<L$ implies 
	\[\frac{1}{C} d(y,z)\le d(g\inv (y) , g\inv (z))\le C d(y,z). \]
Thus for small enough $r>0$ we have
\begin{align}\label{eq:diminq}
\dfrac{\log(\mu(B(g\inv (y),\frac{r}{C})))}{\log \frac {r}{C} + \log C}
\le\dfrac{\log(g_* \mu(B(y,r)))}{\log r }\le
\dfrac{\log(\mu(B(g\inv (y),Cr)))}{\log (Cr)  - \log C}
.\end{align}
Applying the $\limsup_{r\to 0}$ and $\liminf_{r\to 0}$ operators to both sides of \eqref{eq:diminq} yields the desired results.  %and  we have that $\dim(\nu;y)$ is defined  with the desired equality.  
\end{proof}

\subsection{Dimension, entropy, and Lyapunov exponents}
For $x$ a regular point for a $C^{1+\alpha}$ diffeomorphism, we define functions
 \[\gamma_j (x) := \begin{cases}\delta^{r(x)} (x)& j = r(x),\\
 \delta^j(x)- \delta^{j+1}(x)   & u(x)<j< r(x).  \end{cases} \]  
%Note in the case that $\mu$ is ergodic, the functions $\gamma_j(\cdot)$ are a.e.\ constant.
We have the equality
\begin{align}\label{eq:LYent}
h_\mu (f) =
	\begin{cases} 
	\displaystyle\sum_{\lambda_j>0 } \gamma_j \lambda_j &\mu \text{ ergodic,}\\[2em]
	\displaystyle\int \displaystyle \sum_{\lambda_j(x)>0 } \gamma_j(x) \lambda_j (x) \ d \mu(x) \ \ \ &\mu \text{ non-ergodic,}
	\end{cases}
\end{align}
know as the \emph{Ledrappier-Young entropy formula}. 
 The formula \eqref{eq:LYent} was proved in \cite{MR684248} for $C^2$ surface diffeomorphisms, and in  \cite{MR819557} for general $C^2$ diffeomorphisms.  For a statement and proof in the $C^{1+\alpha}$ setting refer to \cite{MR2348606}.

%%%%%%%%%%%%%%%%%%%%%%%%%%%%%%%%%%%%%%%%%%%%%%%%%%%%%%%%%

%%%%%%%%%%%%%%%%%%%%%%%%%%%%%%%%%%%%%%%%%%%%%%%%%%%%%%%%%
\section{Foliation rigidity}
Let $a\colon \T^2 \to \T^2$ be as in the introduction with $\Fols$ and $\Folu$ the stable and unstable foliations.  Before proving the main results we demonstrate mechanisms under which preservation of an $a$-invariant measure forces the preservation of the dynamical foliations $\Folu$ and $\Fols$.    
\subsection{Rigidity of the slow foliation}
Consider  an $a$-ergodic measure $\mu$ with $h_\mu(a) >0$ and $\lambda_\mu^u \neq - \lambda _\mu^s$.  
%Then either $|\lambda_\mu^u| <| \lambda _\mu^s|$ or $|\lambda_\mu^s| <| \lambda _\mu^u|$.  
By the \emph{slow foliation} we mean the foliation whose corresponding Lyapunov exponent is smaller in  absolute value.  We show that, under the additional hypothesis that $\mu$ has full support, any $g\in \Cmu$ preserves the slow foliation.  For simplicity we assume $|\lambda_\mu^u| <| \lambda _\mu^s|$,  and show $g$ preserves $\Folu$.  

\begin{prop}\label{prop:1}
Let $\mu$ be an $a$- ergodic Borel probability measure with full support and $h_\mu(a) >0$.  
Suppose  \[\lambda_\mu^u + \lambda _\mu^s< 0.\]   Then for all $r\ge 1$ we have \[\Cmu \subset \Cmu[\Folu].\]  
\end{prop}

\begin{proof}
Let $g\in \Cmu$.  We write $\Gol = g(\Fol^u)$.  If $\Gol \neq \Fol^u$ then there is some open set $V\subset \T^2$ such that 
\begin{itemize}
\item $V$ is a bifoliation chart for $\Folu$ and $\Fols$;
\item $V$ is a  foliation chart for $\Gol$;
\item  for each $x,y\in V$ the intersection $\Gol_V (x) \cap \locUnst[V] y$ % \Folu_V(y)$ 
contains at most one point, and the intersection is transverse.
\end{itemize}

For $y \in V$ we identify $\locStab[V] y $ with the quotient space $V/\Folu_V$.  Define $\hat \mu$ to be the quotient measure on $\locStab[V] y$ given by \[\hat \mu(B) = \mu\big( \locUnst[V] B\big)\] and define the corresponding pointwise dimensions \[\hat \delta^ +(y) = \limsup _{r\to 0} \dfrac{\log \big(\hat \mu (\locStab[r] y)\big)}{\log r}
\quad \quad \quad 
\hat \delta^ -(y) = \liminf _{r\to 0} \dfrac{\log \big(\hat \mu (\locStab[r] y)\big)}{\log r}.\]  Since the unstable holonomies are bi-Lipschitz, by Proposition \ref{prop:bddRND} $\hat \delta^\pm (y) = \hat \delta^\pm (z) $ for $z\in \locUnst[V] y$.  

By \cite[Lemma 11.3.1]{MR819557} we have
\[\hat \delta^-(y) + \delta^u\le \dim (\mu, y)\] for $\mu$-a.e.\ $y$,  whence, by \eqref{eq:exdim} we conclude that \begin{align}\label{eq:1}\hat \delta^-(y)\le \delta^s \end{align} 
for $\mu$-a.e.\ $y$.

Note the hypothesis $\lambda^u_\mu + \lambda^s_\mu<0$ implies by \eqref{eq:LYent} that $ \delta^u-\delta^s>0$.  
Fix $0<\eta< \delta^u-\delta^s$.  We write $\{\td \mu^\Gol_{V,y}\}_{y\in V}$ for the conditional measures associated to the (measurable) partition of $V$ by the leaves of the local foliation $\Gol_V$.    Note that by Claim \ref{clm:condtransfer}, the fact that $g$ is bi-Lipschitz, and Proposition \ref{prop:bddRND}, we have $\dim(\td \mu^\Gol_{V,y}, y) = \delta^u$ for a.e.\ $y\in V$.  
Define \[\Gamma_R^l :=\{x\in V\mid {l\inv} r^{\delta ^u + \eta} \le \td \mu^\Gol_{V,x}(B(x,r))\le {l} r^{\delta ^u - \eta} \quad \text{for all } 0<r<R\}\]
and fix $l$ and $R$ so that $\mu(\Gamma^l_R)>0$.  On $\locStab[V] y$ define a second quotient measure $\hat \nu^l_R$ by 
\[ \hat \nu^l _R (B) := \mu (\locUnst[V] B \cap \Gamma^l_R).\]
Clearly $ \hat \nu^l_R \ll \hat \mu$ hence, by Proposition \ref{prop:bddRND}, for every $x\in V$ and $\hat\nu^l_R$-a.e.\ $y \in \locStab[V] x$ we have $\underline \dim (\hat \nu^l_R, y) = \hat \delta^- (y)$.  
 
Fix such a $y$.  Using   the uniform transversality of the local foliations $\Gol_V$ and $\Folu_V$ and the fact that the unstable holonomies are bi-Lipschitz, we may find a $1>c>0$ so that 
 \[\locUnst[V]{\locStab[cr] y} \subset \bigcup_{z\in \locUnst[V] y} B_{\Gol}(z, r)\] 
for all sufficiently small $r>0$.  Here $B_{\Gol}(z, r)$ denotes that metric ball of radius $r$ at $z$ in the submanifold $\Gol(z)$.  
 Hence
 \begin{align*}
 \hat \nu^l_R(\locStab[cr] y) &=\int_V \td\mu^\Gol_{V,x}\big( \locUnst[V]{\locStab[cr] y}\cap \Gol_V(x) \cap \Gamma^l_R\big) \ d \mu (x)\\
 &\le \int_V 2  {l} r^{\delta ^u - \eta}  \ d \mu (x)\\
& = Kr^{\delta ^u - \eta}.
 \end{align*}
for some $K$ and all sufficiently small $r>0$.  We thus conclude that $$\underline \dim(\hat \nu ^l_R, y) \ge \delta ^u- \eta >\delta ^s$$  hence 
$\hat \delta^-(y)>\delta ^s$ on a set of positive measure contradicting \eqref{eq:1}.  
\end{proof}

%%%%%%%%%%%%%%%%%%%%%%%%%%%%%%%%%%%%%%%%%%%%%%%%%%%%%%%%%

%%%%%%%%%%%%%%%%%%%%%%%%%%%%%%%%%%%%%%%%%%%%%%%%%%%%%%%%%

\subsection{Rigidity of the fast foliation for equilibrium states}

In the case that  $\mu$ is an equilibrium state, we are able to utilize the local product structure of $\mu$ and obtain a  result stronger than Proposition \ref{prop:1}.

\begin{prop}\label{prop:2}
Let $\mu$ be an equilibrium state for a H\"older continuous potential on $\T^2$ (with respect to the dynamics $a$).  Suppose that $\mu$ is neither  the forwards nor backwards SRB measure and satisfies  \[\lambda_\mu^u \neq -\lambda _\mu^s.\]   Then \[\Cmu \subset \Cmu[\{\Fols, \Folu\}]\] for all $r\ge 1$.

%  \cmt{max entropy}
\end{prop}

\subsubsection{Product structure of equilibrium states}
Before we prove Proposition \ref{prop:2}, we first recall some facts about equilibrium states invariant under uniformly hyperbolic dynamics.  
Let $\Lambda$ be a compact, locally maximal, non-wandering, topologically transitive, hyperbolic set 
for a $C^{1}$ diffeomorphism $f$ of a manifold (i.e. a basic set).   (See, for example, \cite{MR1326374} for relevant definitions.)  Let $\phi\colon \Lambda \to \R$ be a H\"older continuous function and $\mu$ be the associated equilibrium state.  

Recall that $\Lambda$ satisfies a \emph{local product structure}; that is, there exist $0<\delta<\epsilon$ with the property that for all $x,y\in \Lambda$ with $d(x,y)\le 2 \delta$ the intersection\[\locUnst x \cap \locStab y\] contains exactly one point and \[\locUnst x \cap \locStab y\subset \Lambda.\]   For such $x,y$ we write \[[x,y] := \locUnst x \cap \locStab y.\]
Given $x^s\in \locStab [\delta] x$  and $x^u\in \locUnst [\delta] x$ we define the local holonomies 
\begin{align*}&h_{x, x^s} ^s \colon \locUnst[\delta] x\to \locUnst{x^s} &h_{x, x^u} ^u \colon& \locStab[\delta] x\to \locStab{x^u}\\
& h_{x, x^s} ^s \colon z \mapsto [x^s, z]    &h_{x, x^u} ^u \colon &z \mapsto [ z, x^u].
\end{align*}
The following theorem describes a local product structure for equilibrium states.  

\begin{thm}\label{thm:eqstates}
Let $\mu$ be the equilibrium state associated to a H\"older continuous function $\phi$ on $\Lambda$.  Then for each $\sigma \in \{s,u\}$ there exists a family of measures $\{\mu^\sigma_x\}_{x\in \Lambda}$ such that
\begin{enumerate}[label={\alph*)}, labelindent=\parindent, ref=\ref{thm:eqstates}(\alph*)]
\item \label{thm:transmsrunique} the family $\{\mu^\sigma_x\}_{x\in M}$ is uniquely determined up to scalar multiplication and  $\mu^\sigma_x = \mu^\sigma_y$ for $x\in W^\sigma (y)$;

\item $\mu^\sigma_x$ is supported on $W^\sigma(x) \cap \Lambda$ and %if $\mu$ has positive entropy then 
$\mu^\sigma_x (U) >0$ for any non-empty open subset of $W^\sigma(x) \cap \Lambda$;

\item  $f_*\mu^\sigma_x $ and $\mu^\sigma_{f(x)} $ are equivalent with \label{thm:dynjac}
	\begin{align}  
	\label{eq:dynu}\dfrac{d (f_*\mu^u_x)}{d \mu^u_{f(x)}}(f(y)) &= e^{-\phi(y) + P(\phi)}\\ 
	\label{eq:dyns}	\dfrac{d (f_*\mu^s_x)}{d \mu^s_{f(x)}}(f(y)) &= e^{\phi(f(y)) - P(\phi)} \end{align} 
for $y\in W^\sigma(x)$, where $P(\cdot)$ denotes the pressure functional (defined, for example, in \cite{MR1326374}); % and $y \in W^\sigma(x)$;

\item \label{thm:holojac} for $x^s\in \locStab [\delta] x$  and $x^u\in \locUnst [\delta] x$  we have  
\begin{align}
\label{eq:holos} \dfrac { d \mu^u_{x^s}}{d ((h^{s}_{x,x^s})_* \mu^u_x)} (\cdot) &= e^{\omega_{x}^u(\cdot)} \\
\label{eq:holou} \dfrac { d \mu^s_{x^u}} {d ((h^{u}_{x,x^u})_* \mu^s_x)}(\cdot) &= e^{\omega_{x}^s(\cdot)} 
\end{align}
where 
	\begin{align}
	\omega_{x}^u(y)&:= \sum_{i = 0 } ^\infty\phi (f^i([x,y]))  -  \phi (f^i(y)) \label{eq:omegaU} \\
	\omega_{x}^s(y)&:= \sum_{i = 0 } ^\infty   \phi (f^{-i}([y,x]))-  \phi (f^{-i}(y)) \label{eq:omegaS};
	\end{align}
\item  \label{thm:prodstruct}  after suitable normalization, on local charts $[\locStab[\delta]x \cap \Lambda  , \locUnst[\delta]x\cap \Lambda]$ we have the product decomposition
\begin{align}\label{eq:prodstr}d\mu (\cdot) = e^{\omega_{x}^u(\cdot) + \omega_{x}^s(\cdot ) +\phi(\cdot)} \  d (\mu_x^u\times \mu_x^s) ([x,\cdot], [\cdot,x]);\end{align}

\item \label{thm:conditionals} for any measurable partition $\xi$ subordinate to $\Fol^u$, up to normalizing constants, the family 
	 \begin{align*}
	 \{ e^{\omega_{x}^s +\phi}  \mu^u_x\}\label{eq:cond}
	 \end{align*} 
provides a family of conditional probability measures $\td \mu^\xi_x$.

\end{enumerate}
\end{thm}

Complete proofs of Theorem \ref{thm:eqstates} are missing from the literature, but partial proofs and sketches exist.  We contribute another sketch here.

\begin{proof}[Proof sketch of Theorem \ref{thm:eqstates}]
The existence of a  family of measure satisfying \eqref{eq:dynu}--\eqref{eq:holou} is derived in \cite{MR1661262}.  See also \cite[Proposition 2.3]{MR1828477}. 
In fact the properties in Theorem \ref{thm:dynjac} may be taken as defining properties for the family of measures from which \eqref{eq:holos} and \eqref{eq:holou}  are easily derived.  

From \eqref{eq:omegaU} we derive the identity \[\exp(\omega^u_x(y)) = \exp(\omega^u_x([x',y]))\exp(\omega^u_{x'}(y)). \]
By Theorem \ref{thm:holojac} we have
\[ d\mu^u_{x'} ([x', \cdot]) =  \exp(\omega^u_x([x',\cdot])) d\mu^u_{x} ([x, \cdot])\]
hence we verify the expression on the right hand side of \eqref{eq:prodstr} is well defined; that is, the measure is defined independent of the choice of base point $x$.  Furthermore, by Theorem  \ref{thm:dynjac} one verifies that that the measure defined on the right hand side of \eqref{eq:prodstr} is invariant under $f$.  
Indeed we have 
\begin{align*}
d(f_*\mu )(f(y)) &= e^{\omega_{x}^u(y) + \omega_{x}^s(y ) +\phi(y)} \  d (\mu_x^u\times \mu_x^s) ([x,y], [y,x])\\
&= e^{\omega_{x}^u(y) + \omega_{x}^s(y ) +\phi(y)} \  d (f_*(\mu_{x}^u\times \mu_{x}^s) )([f(x),f(y)], [f(y),f(x)])\\
&= e^{\omega_{x}^u(y) + \omega_{x}^s(y ) +\phi(y)} e^{-\phi([x,y]) + \phi(f([x,y]))}\  d (\mu_{f(x)}^u\times \mu_{f(x)}^s) ([f(x),f(y)], [f(y),f(x)])\\
&=e^{\omega_{f(x)}^u(f(y)) + \omega_{f(x)}^s(f(y) ) +\phi(f(y))} \  d (\mu_{f(x)}^u\times \mu_{f(x)}^s) ([f(x),f(y)], [f(y),f(x)])\\
&= d \mu (f(y))
\end{align*}
Theorem \ref{thm:conditionals} then follows from  \eqref{eq:prodstr}. 
\begin{comment}
Indeed on a chart $V = [\locStab[\delta]x \cap \Lambda  , \locUnst[\delta]x\cap \Lambda]$ and for 
$y\in \locStab[\delta]x \cap \Lambda  , z\in  \locUnst[\delta]x\cap \Lambda$ writing \[g(y,z) := e^{\omega_{x}^u([z,y]) + \omega_{x}^s([z,y]) +\phi([z,y])}\] we have

\begin{align*}
\mu (A)&= \int\limits_{\locStab[\delta]x } \int \limits_{ \locUnst[\delta]x} 1_A([z,y]) 
g(y,z) \ d\mu^u_x (y)  \ d\mu^s_x (z) \\
&= \int\limits_{\locStab[\delta]x } \int \limits_{ \locUnst[\delta]x} 
\dfrac{1}{\int \limits_{\xi(y)} g(y',z) \ d\mu^u_x (y') }
\left(\int \limits_{\xi(y)} 1_A([z,y'])  g(y',z) \ d\mu^u_x (y') \right)
g(y,z) \ d\mu^u_x (y)  \ d\mu^s_x (z) \\
&= \int\limits_{\locStab[\delta]x } \int \limits_{ \locUnst[\delta]x} 
\dfrac{1}{c(\xi(y))}
\left(\int \limits_{\xi(y)\cap A} 
 e^{\omega_{x}^u([z,y']) + \omega_{x}^s([z,y']) +\phi([z,y'])}
\ d\mu^u_x (y') \right)
g(y,z) \ d\mu^u_x (y)  \ d\mu^s_x (z) \\
\end{align*}
\end{comment}
\end{proof}

We note that the uniform hyperbolicity of $f$ and the H\"older continuity of $\phi$ ensures $\omega_{x}^\sigma(y)$ is well defined; furthermore we have that $\omega_{x}^\sigma(y)$ is continuous in both arguments $x$ and $y$.

By Claim \ref{clm:condtransfer} and Theorem \ref{thm:conditionals}, for a.e.\ $x\in \Lambda$ we expect 
\[f_*(e^{\omega_{x}^s +\phi}  \mu^u_x) = K e^{\omega_{f(x)}^s +\phi\circ f}  \mu^u_{f(x)}\]
for some constant $K$.  We check that $K= e^{P(\phi) +\phi(f(x))}$ works.   Note that even for $x'\in \unst x\cap \Lambda$ the measures $e^{\omega_{x}^s +\phi}  \mu^u_x$ and $e^{\omega_{x'}^s +\phi}  \mu^u_{x'}$ differ by the constant factor of $e^{\omega_x^s(x')}$, hence it is expected that the rescaling $K$ will depend on the point $x$.

\subsubsection{Proof of Proposition \ref{prop:2}}
We return to the setting where $\Lambda = \T^2$, $a$ is an Anosov diffeomorphism, and $\mu$ is assumed to be an equilibrium state for a H\"older continuous potential $\phi$. 
By passing to $a\inv$ if necessary we may assume that $|\lambda^u_\mu|<|\lambda^s_\mu|$ whence Proposition \ref{prop:1} implies that any $g\in \Cmu$ preserves $\Folu$.  We use the local product structure of $\mu$ to show that $g$  preserves $\Fols$ under the additional assumption that $\mu$ is not the forwards SRB measure.  

Fix an $r\ge 1$ and $g\in \Cmu$.   We write $\Gol = g(\Fols)$.  
Using that the stable holonomies are bi-Lipschitz and that the Radon-Nikodym Derivatives in Theorem \ref{thm:holojac} are bounded we obtain $$\dim (\mu^u_x, x) =\dim (\mu^u_{x'}, x')$$ for any $x'\in W^s(x)$, assuming $\dim (\mu^u_x, x)$ is defined.  That is, the set of points on which $\dim (\mu^u_x, x)$ is constant is $\Fols$-saturated.  We show this set of points is also $\Gol$-saturated via the following claim.

\begin{claim}\label{claim:gpreservesCond}
 %Let $\mu$ be an equilibrium state and 
Let $g\in \Cmu$.  Then for every $x$ there exists $K>0$ so that 
\begin{align*}\label{eq:gjac}\dfrac{d (g_* \mu^u_{g\inv x})}{d \mu^u_x }(y) = K \dfrac{\exp( \omega^s_x(y))+ \phi)}{\exp(\omega^s_{g\inv (x)}(g\inv(y)) + \phi (g\inv(y)) )}.\end{align*}
In particular, $\mu^u_x $ is equivalent to $g_* \mu^u_{g\inv x}$ with Radon-Nikodym derivative bounded away from $0$ and $\infty$. 
\end{claim}

\begin{proof}
We continue to write $\Gol = g(\Fols)$.  Fixing an $x\in \T^2$ we may find a $\delta>0$ small enough so that for the local chart \[V =[\locStab[\delta] x , \locUnst[\delta] x ]\]
and all $y,z\in V$ the intersection \(\Gol_V(y) \cap \locUnst[V] z\) contains at most one point.  Fix an open set $U\subset \locUnst[\delta] x$ and let \[T := \Gol_V(U).\]  For $y\in V$ 
write $T_y : = T \cap \locUnst[V]y$.   See Figure \ref{fig:1}.  

\begin{figure}[h]\label{fig:1}
	\begin{center}
	\begin{pspicture}(-2,-2.2)(2,2.3)
	\psset{unit=1.01in}
%	  \psaxes[linecolor=gray,labels=none,]{-}(0,0)(-1,-1)(1,1)
%\psbezier[linewidth=1pt]{->}(0,0)(0,.5)(.3,4)(.3,6) \psbezier[linewidth=1pt]{->}(0,0)(0,-.5)(-.3,-4)(-.3,-6)
%\psbezier[linewidth=1pt]{->}(0,0)(4,0)(4,.5)(6,.5)
%	 {\parametricplot[algebraic,arrows=-,linewidth=1.2pt]{-1}{1} {tan(t)|.17*t }}

\def\vh{.05}
\def\frsize{.5pt}
\def\grey{.4}
\definecolor{MyColor}{rgb}{\grey,\grey,\grey}

%%%    Frame
\pscustom[linewidth=\frsize,showpoints=true]{
	\pscurve(-1,-1)(-1.003,-.85)(-1,-0.77655048)(-.99,-.4)(-1,\vh)(-1,1) 
	\pscurve[liftpen=1,showpoints=true](-1,1)(-.4,1.03)(\vh,1)(.4,.973)(.85,1.0110819)(.95,1.0110819)(1,1) 
	\pscurve[liftpen=1](1,1)(1.005,.8)(1, 0.612415196)(.993,.4)(1,\vh)(1.02,-.4)(1,-1) 
	\pscurve[liftpen=1](1,-1)(.4,-.98)(\vh,-1)(-.4,-1.02)(-.65,-1.0110819)(-.85,-1.0050)(-1,-1) 
	}

%%%   T
\pscustom[linewidth=1pt,,fillstyle=solid,fillcolor=lightgray]{
 \parametricplot[algebraic,arrows=-]{-.8}{1.05} {t -.2 |  .04*sin(1.5*3.14*t) +t}
	\pscurve[liftpen=1](.85,1.0110819)(.95,1.0110819)(1,1) 
	\pscurve[liftpen=1](1,1)(1.005,.8)(1, 0.612415196)
	
 \parametricplot[liftpen=1,algebraic,arrows=-,]{.6}{-1.05} {t +.4 |  .04*sin(1.5*3.14*t) +t}
	\pscurve[liftpen=1](-.65,-1.0110819)(-.85,-1.0050)(-1,-1) 
	\pscurve[liftpen=1](-1,-1)(-1.003,-.85)(-1,-0.77655048)
}

%%%% G-foliation
\multido{\N=-1.2+.1}{24}{	 	 {\parametricplot[algebraic,arrows=-,linecolor=MyColor,linewidth=.3pt]{-1.1}{1.1} {t +\N |  .04*sin(1.5*3.14*t) +t}}}

%%%% U and T_y

  %%W^u(y)
	 {\parametricplot[linewidth = \frsize, algebraic,arrows=-]{-1}{1} {t|.05*t^2 + .02*sin(3.14*t)+.5}}

\pscustom[linewidth=1.5pt]{
	 {\parametricplot[algebraic,arrows=-]{-.21}{.42} {t|.05*t^2 + .02*sin(3.14*t) }}
	  {\parametricplot[liftpen=2,linewidth = .8pt, algebraic,arrows=-]{.29}{.93} {t|.05*t^2 + .02*sin(3.14*t)+.5}}
}

%%%  Cross  %%%%
\pscustom[linewidth=\frsize]{

	 {\parametricplot[liftpen=2,algebraic,arrows=-]{-1}{1} {t|.05*t^2 + .02*sin(3.14*t) }}
	 {\parametricplot[liftpen=2,algebraic,arrows=-]{-1}{1} {.05*t^2 + .02*sin(3.14*t) |t}}
}

%%%%   Labels
	\psdots[dotsize=3pt](0,0)(-.12,.495)

	\uput{.031}[130](0,0){$x$}
	\uput{.031}[130](-.12,.495){$y$}
	\uput{.031}[60](.1,0){$U$}

	\uput{.031}[0](-.5,-.5){$T$}
	\uput{.031}[60](.6,.53){$T_y$}

	\uput{.031}[50](-.83,.51){$W^u_V(y)$}
	\uput{.031}[50](-.83,0){$W^u_\delta(x)$}
	\uput{.031}[0](0,-.83){$W^s_\delta(x)$}

	\uput{.031}[0](1.4,.83){$\mathcal G$}

%	 {\parametricplot[algebraic,arrows=-,linewidth=1.2pt]{-1}{1} {t|- t^5-.12 *t}}

%	 {\parametricplot[algebraic,arrows=-,linewidth=1.2pt]{-4}{4} {t |-sin(t)+ .3*t }}

	\end{pspicture}
	\end{center}
	\caption{The local chart $V$} \label{fig:1}
\end{figure}

Using the continuity the Radon-Nikodym derivatives in \eqref{eq:holos}  % Theorem \ref{thm:holojac}
 and the fact that the measures $\mu^u_x$ are non-atomic, we have that each of the maps
\begin{align*}
j_{1, U} &\colon y\mapsto \int_{T_y} \exp( \omega_{y}^s +\phi)\ d \mu^u_y;\\
j_{2, U} &\colon y \mapsto \int_{T_y} \exp(  \omega^s_{g\inv (y)}\circ g\inv+\phi \circ g\inv) \ d g_*(\mu^u_{g\inv(y)});\\
c_1&\colon y\mapsto \int_{\locUnst[V] y} \exp( \omega^s_y+\phi)\ d \mu^u_y; \\
c_2&\colon y\mapsto \int_{\locUnst[V] y} \exp(  \omega^s_{g\inv (y)} \circ g\inv +\phi\circ g\inv) \ d g_*(\mu^u_{g\inv(y)})
\end{align*}
is a continuous map $V \to \R^{}$; furthermore, $c_1$ and $c_2$ are bounded away from $0$ and $\infty$.
By Claim \ref{clm:condtransfer} and  Theorem \ref{thm:conditionals} the equality  \begin{align}\dfrac{j_{1, U}(y)}{c_1(y)} = \dfrac{j_{2, U}(y)}{c_2(y)}\label{eq:44}\end{align}holds for every $y$ in a  subset of $ V$ of full measure.  
The fact that the measure $\mu$ has full support implies that \eqref{eq:44} holds on a dense subset of $ V$. 
We thus obtain  \[{j_{2, U}(y)} = \dfrac{c_2(y)}{c_1(y)}{j_{1, U}(y)}\]  for \emph{every} $y\in V$ since each of the functions $j_{1, U}, j_{2, U} $ and  $\dfrac{c_2(y)}{c_1(y)}$ is continuous. Since the open set $U$ was arbitrary, this establishes the claim with $K(x) = \dfrac{c_2(x)}{c_1(x)}$.
\end{proof}
Note that the constant $K(x) = \dfrac{c_2(x)}{c_1(x)}$ is independent of the choice of the local chart $V$  due to  Theorem \ref{thm:conditionals} and Claim \ref{clm:condtransfer}.

To complete the proof of Proposition \ref{prop:2} we choose a measurable partition $\xi$, subordinate to the foliation $\Folu$, and comprised of half open curves.  More precisely for each $x\in \T^2$ there is an embedding \[\gamma_x\colon [0,1)\to W^u(x)\] with $\xi(x) = \gamma_x\big([0,1)\big)$.  Call $\gamma_x\big((0,1)\big)$ the \emph{interior} of $\xi(x)$ and $\gamma_x(0)$ the \emph{endpoint} of $\xi_x$.  We may additionally choose $\xi$ so that the set of endpoints has measure zero.  
 
 For the partition $\xi$, we fix the family of conditional measures $\{\td \mu_x^\xi\}$ given by an appropriate scaling of the family $\{e^{\omega_{x}^s +\phi}  \mu^u_x  \}$  as guaranteed by Theorem \ref{thm:conditionals}.
Since the function $ e^{\omega_{x}^s+ \phi}$ is locally bounded away from $0$ and $\infty$,  for $y$ in the interior of $\xi(y)$ we have \[\dim(\td \mu^\xi_y, y) = \dim (\mu^u_y, y).\] %For $y$ an endpoint $\xi(y)$ we have \[\dim(\td \mu^\xi_y, y) \ge \dim (\mu^u_y, y).\]

%Again we write $\Gol =g(\Fols)$. 
Supposing $\Gol\neq \Fols$, we may find an $x\in \T^2$ and a $\delta>0$ so that for $V = [\locStab[\delta] x, \locUnst[\delta] x]$ and all $y,z\in V$ the intersections
\[\Gol_V(z) \cap \locUnst[V] y \quad \text{and} \quad\Gol_V(z) \cap \locStab[V] y\]
contain at most one point and are transverse.  

Let $\Upsilon\subset V$ denote the set of all points such that $\dim(\mu^u_x, x) = \delta^u$.  Then  $\Upsilon$ has full measure in $V$.  
By \eqref{eq:holos} and the argument preceding Claim \ref{claim:gpreservesCond} we see that $\Upsilon$ is $\Fol^s_V$-saturated.  On the other hand, Claim \ref{claim:gpreservesCond} and a similar argument ensures that $g\inv(\Upsilon)$ is $\left(\Fols_{g\inv (V)}\right)$-saturated, whence  $\Upsilon$ is $\Gol_V$-saturated.
Thus $\Upsilon$ contains an open set; in particular $\Upsilon$ contains a curve $I\subset \locUnst [\delta] x$.
We may further assume that $I$ is contained in the interior of $\xi(x)$ whence we obtain that for \emph{every} $y\in I$ 
\[\dim(\td \mu^\xi_y, y) = \dim (\mu^u_y, y) = \delta^u.\]  
By  Proposition \ref{prop:Haus} we obtain that $1 = \dim_{H}(I) \le \delta^u$, where $\dim_{H}(I)$ denotes the Hausdorff dimension of the set $I$.  We thus obtain contradiction unless $\mu$ is the forwards SRB measure for $a$.  
\qed

%%%%%%%%%%%%%%%%%%%%%%%%%%%%%%%%%%%%%%%%%%

%%%%%%%%%%%%%%%%%%%%%%%%%%%%%%%%%%%%%%%%%%

\section{Proof of Theorem \ref{thm:2}}
\newcommand{\E}{\mathcal E}

Recall our fixed notation: $a\colon \T^2 \to \T^2$ a $C^{\theta}$ Anosov diffeomorphism,  $h\colon \T^2 \to \T^2$ bi-H\"older, and $A \in \GL(2, \Z)$ such that \[h\circ a \circ h\inv =  L_A.\]  We fix $\mu$ as in Theorem \ref{thm:2}, and $r\ge 1+\alpha$ for some $\alpha>0$.  
By passing to $a\inv$ if necessary, we assume $|\lambda_\mu^u|\le |\lambda_\mu^s|$.  We continue to write $\Folu$ and $\Fols$ for the foliations of $\T^2$ induced by the dynamics of $a$.  For $g\in \Cmu$,   Proposition \ref{prop:1} guarantees $g$ preserves $\Folu$.  
%Note that if $h_\mu(g) = 0$ for all $g\in \Cmu$ then there is nothing to prove; hence we assume otherwise. Under these assumptions, 
We show the entropy $h_\mu(g)$ is effectively computed by the dynamics of $g$ along the foliation $\Folu$. %; we note this requires $r\ge 1+\alpha$.

First note that by Corollary 15.4.2 of \cite{MR2348606} and the fact that $\mu$ has no atoms, for any $C^{1+\alpha}$ $\mu$-preserving diffeomorphism $g\colon \T^2 \to \T^2$ and $\mu$-a.e.\ regular point $x$, either
\begin{itemize}
\item[] $r(x) = 0$ and $\lambda_0(x) = 0$, or 
\item[] $r(x) = 1$ and $\lambda_0(x) \cdot \lambda_1(x) \le 0$.
\end{itemize}
In other words, there is no positive measure set for which the Lyapunov exponents for $g$ are defined and are all positive or all negative.  For $g\in \Cmu$ we notate by $\Lambda(g)$ the set of regular points under $g$ whose Lyapunov exponents are not all positive or all negative.  The functions $r$ and $\lambda_j$ will be as in Section \ref{sec:lyap} with respect to the dynamics of $g$.  We write $E^i_g(x)$ and $\wtd W^i_g(x)$ for the Lyapunov subspaces and corresponding Pesin manifolds at $x$ under the dynamics of $g$.

For $g\in \Cmu$,  define a bounded measurable function $\chi_g$ on $\T^2$ by
	 \[\chi_g\colon x \mapsto \limsup_{n\to \infty} \dfrac{1}{n} \log\left(\|Dg^n _x v\|\right)\] 
	 where $v\in T_x\Folu(x)\sm\{0\}$.
Defining the function $J_g$ on $\T^2$ by \begin{align}\label{eq:Jg}J_g\colon x \mapsto \dfrac{\|Dg _x(v)\|}{\|v\|}\end{align} for $v\in T_x\Folu(x)\sm\{0\}$, we alternatively have  
\begin{align*}\chi_g(x) = \limsup_{n\to \infty} \dfrac{1}{n}\sum_{i= 0}^ {n-1} \log(J_g (g^i(x))).\end{align*}	 

For $g\in \Cmu$ let $\I_g$ denote the $\sigma$-algebra of $g$-invariant sets.  By the Birkhoff Ergodic Theorem (see, for example, \cite{MR1326374}) we have for $\mu$-a.e.\ $x$  the equalities
\begin{align}
\chi_g(x)& = \Exp(\log J_g\mid \I_g) (x) \notag\\
& =\lim _{n\to \pm \infty }  \dfrac{1}{n} \log\left(\|Dg^n _x v\|\right) \label{eq:ShowsLyap}.
%& =\lim _{n\to \infty  \lim_{n\to \infty} \dfrac{1}{n} \log\left(\|Dg^n _x v\|\right)
\end{align}	
Here, the right hand side of the first equality denotes a conditional expectation.  % of $\log J_g$ conditioned on $\I_g$.  
In particular \eqref{eq:ShowsLyap} shows that  $\chi_g(\cdot)$ is a Lyapunov exponent, whence,  for $\mu$-a.e.\ $x\in \Lambda(g)$ with $r(x) = 1$, we have that $T_x \Folu (x)$ is a Lyapunov subspace.   Indeed if $x\in \Lambda(g)$ is a regular point with $r(x) = 1$ and $0\neq v\in T_x \Folu (x)$ satisfies \[v = \alpha_0 v_0 + \alpha_1 v_1\] for $v_j \in E^j_g(x)$ and $\alpha_j \neq 0$ then we have 
  \[\lambda_0 (x) = \lim_{n\to -\infty} \dfrac{1}{n} \log\left(\|Dg^n _x v\|\right)  \neq  \lim_{n\to \infty} \dfrac{1}{n} \log\left(\|Dg^n _x v\|\right)  = \lambda_1(x)\] which can only hold on a null set by \eqref{eq:ShowsLyap}.  
  Let $i(x)$ be the a.e.-defined $\{0,1\}$-valued function on $\Lambda(g)$  satisfying $\chi_g(x) = \lambda_{i(x)}(x)$.

\newcommand{\Expo}{\mathrm{Exp}}

\begin{claim}\label{clm:StabFol}
For $x\in \Lambda(g)$ with $i(x)$ defined and $\lambda_{i(x)} (x)\neq 0$  we have that $\wtd W^{i(x)}_g(x) \cap \Folu(x)$ contains a neighborhood of $x$ in $\Folu(x)$.  
\end{claim}

\newcommand{\ball}[2][r]{B_{#2}(0,{{#1}})}
\begin{proof}
Fix a Riemannian metric on $\T^2$ and let $$\Expo'_y\colon T_y \Folu(y) \to \Folu(y)$$ denote the exponential map for the restriction of the metric to $\Folu(y)$.  We have that $\Expo'_y$ is $C^{1+\alpha}$.   
We notate by  $\ball y \subset  T_y \Folu(y)$ the norm-ball of radius $r$ centered at zero in $T_y \Folu(y)$.  For a fixed $r$ we then have that the maps $\Expo'_y\colon \ball y \to \Folu(y)$ are bi-Lipschitz with Lipschitz constants bounded uniformly in the variable $y$.  

For $r$ sufficiently small define $\td g\colon\ball y \to T_{g(y)}  \Folu(g(y))$ by $$\td g =(\Expo'_{g(y)})\inv\circ g \circ \Expo'_y.$$    Using a smooth bump function, we may build $G_y\colon T_y \Folu (y)\to T_{g(y)}  \Folu(g(y))$ with 
$$G_y (v) = \begin{cases} \td g(v) & \|v\| \le r,\\
Dg_y (v)& \|v\|\ge 2r.
\end{cases}$$
We have that $G_y$ is a Lipschitz perturbation of $Dg_y$:  $$ \| (Dg_y - G_y) (v)-(Dg_y - G_y) (u)\| \le \gamma_r \|v-u\|.$$
Furthermore $Dg_y(0) = G_y(0) = 0$ by construction, hence 
$$ \| (Dg_y - G_y) (v)\| \le \gamma_r \|v\|.$$
Furthermore, by taking $r$ sufficiently small, we may make $\gamma_r$ arbitrarily small.  We emphasize that the above bounds are uniform over all $y \in \T^2$.  We write $G^n_y:= G_{g^n(y)} \circ G_{g^{n-1}(y)} \circ \dots \circ G_{y}$.  

Now let $x$ be as in the claim. By passing to $g\inv$ if necessary we may assume $i(x) = 0$, that is, $\lambda_1(x)\ge 0>\lambda_0(x)$.  Fix some $0<\epsilon <\frac{1}{4}| \lambda_0(x)|.$ % - \lambda_0(x))$.   
\begin{comment}
Then there is a function $C(x) = C_\epsilon(x)$ so that for any $v\in T_x \Folu(x) $
$$\|Dg^n_x v\|\le C e^{\lambda_0(x) + \epsilon} \|v\|$$
and $$C(g^{|n|}(x)) \le C(x) e^{\epsilon |n|}$$ for any $n\in \Z$.  
(See, for example, \cite{MR2348606}).
We inductively calculate that for $n\ge 0$ $$\|G_x^n(v)\|\le \left(C(x) e^{\epsilon} e^{\lambda_0(x) +\epsilon} +\gamma_r\right)^n\|v\|$$
\end{comment}
The non-uniform hyperbolicity of $Dg$ along the orbit of $x$ guarantees we may find a constant $C= C(x,\epsilon)$ (where $C(x,\epsilon)$ depends measurably on $x$) so that for $v\in T_x \Folu(x) $
$$\|Dg_x^n v\|\le C e^{n(\lambda_0(x) + \epsilon)} \|v\|.$$

We write $\eta = e^\epsilon-1>0$.  For $ v\in T_y \Folu(y)\sm \{0\}$, we may choose $r$ small enough so that 
$$\gamma_r<\eta \cdot \inf\left \{ J_g(y) \mid y \in \T^2\right\} $$ where $J_g$ is as in \eqref{eq:Jg}.
The bound on $\gamma_r$ then guarantees that for any $R$ and $y\in \T^2$ 
$$G_y(\ball[R] y) \subset Dg_y\left( \ball[(1+\eta) R] y\right).$$
Indeed, for any $ v\in B_y(0, R)$ %T_y \Folu(y)$
\begin{align*}
\|G_y(v)\| &\le \|Dg_y(v)\| + \gamma \|v\|\\
&\le J_g(y) R + \gamma R\\
&\le (J_g(y) +\eta J_g(y)) R\\
\end{align*}
so $G_y(v) \subset Dg_y (B_y(0, (1+\eta)R)$ for all $y\in M$.

%(where $(1+\eta)\cdot A$ denotes the scaling of all vectors in $A$ by a factor of $1+\eta$.)  
Consequently, we obtain
$$G_y^n(\ball[R] y) \subset Dg_y^n\left(\ball[(1+\eta)^n R] y\right). $$
(We note the above argument works because $\dim T_y \Folu(y) = 1$ and thus the norm and co-norm of $Dg_y$ are equal at every point $y$; a higher dimensional arguement would require far more subtle  control of the geometry.)

Thus for $v\in T_x \Folu(x) $ we have
$$\|G^n _x (v)\|\le C e^{n(\lambda_0(x) +  \epsilon)}(1+\eta)^n \|v\| = C e^{n(\lambda_0(x) + 2 \epsilon)}\|v\| .$$  
In particular there is some $r'>0$ so that $G^n_x(\ball[r'] x) \subset \ball[r ] {g^n(x)}$ for all $n\ge 0$. 
 Let $$ U =  \Expo_x'(\ball[r'] x ).$$ 
We then have that $U\subset \Folu(x)$ and for $y\in U$ \[d(g^n(x), g^n(y)) \le C' e^{n(\lambda_0(x) + 2 \epsilon)}d(x,y) \] for some $C'$.  This characterizes $U$ as a local stable Pesin manifold for $\lambda_0$ at $x$ and the claim follows.  \end{proof}

For $g\in \Cmu$ let $\Omega(g)\subset \Lambda(g)$ denote the set regular points with no Lyapunov exponent equal to zero.  We partition $\T^2$ into 4 measurable sets $\{{\Upsilon_0 :=\Omega(g)}, \Upsilon_1, \Upsilon_2, \Upsilon_3\}$ where $\Upsilon_1 $, $\Upsilon_2$, $\Upsilon_3$ are the subsets of $\Lambda(g)\sm \Omega(g)$ containing, respectively, one positive, one  negative, or only one (and hence zero) Lyapunov exponent.  Let $\nu_i$ denote the restriction of $\mu$ to $\Upsilon_i$
\[ \nu_i (A) = \mu(A \cap\Upsilon_i).\]
Applying the Ledrappier-Young entropy formula \eqref{eq:LYent} to either  $g$ or $g\inv$ it follows that
\begin{align}\label{eq:ijk}
h_{\nu_i} (g) &=\begin{cases} h_\mu(g)& i = 0,\\[2mm]
0& 1\le i\le 3.
\end{cases}
\end{align} %for $g\in \Cmu$, $r\ge 1+\alpha$.  
That is, the entropy $h_\mu(g)$ is entirely concentrated on the set $\Omega(g)$.  In particular, \eqref{eq:LYent} and Claim \ref{clm:StabFol} imply that for $\mu$-a.e.\ $x\in \Lambda(g) \sm \Omega(g)$ we have $\chi_g(x) = 0$.  

We write $\E^u$ for the unstable linear foliation on $\T^2$ induced by the dynamics of $ L_A$.  
Let $\td \E^u$ denote the pull-back of $\E^u$ to  $\R^2$.  
Note that the quotient space $\R^2/\td \E^u$ may naturally be identified with the $1$-dimensional linear space $ \R^2/\td\E^u(0)\cong \R$.  
Note the homeomorphism $h\circ g \circ h\inv$ preserves the foliation $\E^u$ for $g\in \Cmu$.  Furthermore,
\begin{claim}\label{claim:transAffine}
$h\circ g \circ h\inv$ acts as an affine map transverse to $\E^u$:  any lift of $h\circ  g \circ h\inv$ to $\R^2$ induces an affine action on the quotient $\R^2/\td \E^u\cong \R$.  
\end{claim}
\begin{proof}
Let $\td l\colon \R^2 \to \R^2$ be a lift of $h\circ  g \circ h\inv$.  Choose any $x\in \R^2$ and $y\in \R^s\sm \td \E^u(x)$ and let $\eta = \dfrac{\rho(\td l(\td \E^u(x)) , \td l(\td \E^u(y)))}{\rho(\td \E^u(x), \td \E^u(y))}$ where $\rho$ denotes Euclidean distance.  Since the leaves of $\E^u$ are linear and dense in $\T^2$ we deduce that 
\begin{align}\label{eq:cow} \dfrac{\rho(\td l(\td \E^u(x')) , \td l(\td \E^u(y')))}{\rho(\td \E^u(x'), \td \E^u(y'))}= \eta\end{align} for any $x',y' \in \R^2$ with $\rho(\td \E^u(x), \td \E^u(y))= \rho(\td \E^u(x'), \td \E^u(y'))$.  \eqref{eq:cow} can then be shown to hold for any $x',y' \in \R^2$ and the result follows.
\end{proof}

Note that for any two lifts of $h\circ g\circ h\inv$ to $\R^2$ the induced maps on $\R^2/\td \E^u$ differ only by a translation.  For $g\in \Cmu$ we write $\Psi(g)$ for the linear component of the affine map on $\R^2/\td \E^u$ induced by a lift of $h\circ g\circ h\inv$.

\begin{claim}\label{clm:subset}
Let $g\in \Cmu$ with $\Omega(g) \neq \emptyset$.  Then either 
\begin{itemize}
\item[] $\wtd W^1_g(x)\subset \Folu(x)$ for all $x\in \Lambda(g)$ with $\lambda_1(x) >0$; or 
\item[] $\wtd W^0_g(x)\subset \Folu(x)$ for all $x\in \Lambda(g)$ with $\lambda_0(x) < 0$. 
\end{itemize}
\end{claim}
\begin{proof}
Let $x$ be in  $\Omega(g)$.  By passing to $g\inv$ we may assume $\wtd W^0_g(x)$ is transverse to $ \Folu(x)$.  We deduce the first case of the conclusion under these assumptions.

Let $\td l\colon \R^2 \to \R^2$ be a lift of $h\circ  g \circ h\inv$.  Then by Claim \ref{claim:transAffine} for any $x,y\in \R^2$ we have \[\rho(\td l^n(\td \E^u(x)) , \td l^n(\td \E^u(y))) \to 0 \quad \text{as } n\to \infty.\]
If there were any point $x'\in \T^2$ with $\wtd W^1_g(x')\not \subset \Folu(x')$, we would be able to find $x\in \R^2, y \in \R^2\sm \td \E^u(x)$ with  \[\rho(\td l^{-n}(\td \E^u(x)) , \td l^{-n}(\td \E^u(y))) \to 0 \quad \text{as } n\to \infty\] contradicting Claim \ref{claim:transAffine}.  
\end{proof}

\def\bchi{\overline \chi}
Note in particular that either $\chi_g(x)\ge 0$ for $\mu$-a.e.\ $x$ or $\chi_g(x)\le 0$ for $\mu$-a.e.\ $x$.
We define the  function	\[\bchi\colon \Cmu \to \R\]  by \[\bchi\colon g\mapsto \int \chi_g \ d \mu.\]
From  \eqref{eq:LYent}, \eqref{eq:ijk}, Claim \ref{clm:subset}, and the observation that $\chi_g(x) = 0$  for $\mu$-a.e.\ $x\in \Lambda(g)\sm \Omega(g)$ it follows that for $g\in \Cmu$ \begin{align}\label{eq:EntChi} h_\mu(g) = |\bchi(g)| \delta^u = \int |\chi_g(x)| \delta^u \ d \mu(x).\end{align}

We show that $\bchi$ is a homomorphism from $(\Cmu, \circ)$ to $(\R, {+})$:
\begin{align}\label{eq:linFun} \bchi (g_1\circ g_2 ) =  \bchi (g_1) +  \bchi( g_2) .\end{align}
Indeed, with our previous notations we have
$$\chi_g(x) = \lim_{n\to \infty} \dfrac{1}{n}\sum_{i= 0}^ {n-1} \log(J_g (g^i(x)))$$
for $\mu$-a.e.\ $x$ and 
\[\bchi(g) :=\int \chi_g = \int \Exp(\log(J_g) \mid \I_g) = \int \log(J_g )\]
whence 
{
\allowdisplaybreaks\begin{align*}
\bchi(g_1&\circ g_2):= \int \chi_{g_1 \circ g_2} \\
%&= \int \Exp(\chi_{g_1 \circ g_2}\mid \I_{g_1\circ g_2} )\\
&= \int \lim_{n\to \infty} \dfrac{1}{n}\sum_{i= 0}^ {n-1} \log\left(J_{g_1 \circ g_2} \big((g_1\circ g_2)^i(x)\big)\right)\\
&= \int \lim_{n\to \infty} \dfrac{1}{n}\sum_{i= 0}^ {n-1} \left(\log\left(J_{g_1} \circ g_2 \big((g_1\circ g_2)^i(x)\big)\right) + \log\left(J_{ g_2} \big((g_1\circ g_2)^i(x)\big)\right)\right)\\
%&= \int \lim_{n\to \infty} \dfrac{1}{n}\sum_{i= 0}^ {n-1} \log\left(J_{g_1} \circ g_2 \big((g_1\circ g_2)^i(x)\big)\right) + \lim_{n\to \infty} \dfrac{1}{n}\sum_{i= 0}^ {n-1} \log\left(J_{ g_2} \big((g_1\circ g_2)^i(x)\big)\right)  \\
&= \int \Exp (\log(J_{g_1}\circ g_2) \mid \I_{g_1 \circ g_2} )+ \Exp (\log(J_{g_2}) \mid \I_{g_1 \circ g_2} )\\
&=\int \log(J_{g_1}\circ g_2) + \int \log(J_{g_2})\\
&=\bchi(g_1) + \bchi(g_2).
\qedhere
\end{align*}}\ignorespaces 

In particular,  \eqref{eq:EntChi}  and \eqref{eq:linFun} show that for $\mu$ as in Theorem \ref{thm:2} and $r\ge 1+\alpha$,  the metric entropy satisfies a `signed additivity' property on $\Cmu$: 
\begin{align} h_\mu(g_1 \circ g_2) = \begin{cases} 
	 h_\mu(g_1) + h_\mu(g_2) & \bchi(g_1) \cdot \bchi(g_2) \ge 0, \\[.5em]
	  |h_\mu(g_1) - h_\mu(g_2) | & \text{otherwise,}\\
\end{cases}\label{eq:entadd}
 \end{align}
for $g_1, g_2 \in \Cmu$.

\eqref{eq:entadd} establishes part (1) of Theorem \ref{thm:2}.  Note that if $\Cmu$ contains no positive entropy diffeomorphisms Theorem \ref{thm:2} follows.  We thus may assume there exists some element $g\in \Cmu$ with $h_\mu(g)>0$.

\begin{claim}
Let $g_1, g_2 \in \Cmu$ be such that the linear maps $\Psi(g_1) $ and $\Psi(g_2)$ are equal.  Then $h_\mu(g_1) = h_\mu(g_2)$.  
\end{claim}
\begin{proof}
We have that $\Psi(g_2 \circ g_1\inv)$ is the identity whence $h_\mu(g_2 \circ g_1\inv) = 0$. % by similar arguments as in the proof of Claim \ref{clm:subset}.  
The result then follows from  \eqref{eq:entadd}.  
\end{proof}

Now for $g\in \Cmu$ fix a lift $\td l\colon \R^2 \to \R^2$ of the homeomorphism $h\circ  g \circ h\inv\colon \T^2\to \T^2$.  Let $v= \td l(0)$.   Then the map $x\mapsto \td l(x)-v$ preserves the lattice $\Z^2$ and the linear map it induces on $\R^2/\td \E^u$ is equal to $\Psi(g)$.  Let $L\colon \R^2 \to \R^2$ be the unique linear extension of the action of $x\mapsto \td l(x)-v$ restricted to the lattice $\Z^2$.  
By the density of leaves of $\E^u$ on $\T^2$ the linear action induced by $L$ on $\R^2/\td \E^u$ is also equal to $\Psi(g)$.

Since $L$ and $A$ both descend to automorphisms of $\T^2$ which preserve the one-dimensional subgroup $\E([0])$, they commute.
We have that the centralizer of $A$ in $\Gl(2,\Z)$ is of the form \[C(A) = \{ \pm M^n \mid n\in \Z\}\]  for some hyperbolic matrix $M$  (see, for example, \cite{MR1449997}).  
Hence $L = \pm M^n$ for some $n$; in particular, for any $g\in \Cmu$ the linear map  $\Psi(g)$  is equal to the map induced by $ \pm M^n$ on $\R^2 /\td \E$ for some $n\in \Z$.  Consequently we obtain that there is a smallest positive entropy for all diffeomorphism in $\Cmu$.
Indeed for any $g\in \Cmu$ with $h_\mu(g) >0$ we have that $\Psi(g) $ is equivalent to the map induced by $\pm M^n$ for some $n$, thus a (non-strict) lower bound on the entropy of any positive entropy  map in $\Cmu$ is \[\dfrac{1}{|n|}h_\mu (g).\]
We check the above lower bound is in fact independent of the choice of $g$.  Let $g'$ be such that $h_\mu(g')>0$ and $\Psi(g')$ is equivalent to the map induced by $\pm M^{n'}$ for some $n'$.  Then we have 
$h_\mu((g')^n) = h_\mu(g^{n'})$ hence
\[\dfrac{1}{|n|}h_\mu (g) = \dfrac{1}{|n'|}h_\mu (g').\]
 It then follows that the set
  $$\{h_\mu(g)\mid g\in \Cmu\}$$
   is discrete.

Let $\lambda$ denote the smallest positive entropy attained  by any map in $\Cmu$ and let  $g\in \Cmu$ be so that $h_\mu(g)= \lambda$.  Then the image of $g$  generates the subgroup $\Cmu/N$ in statement (2) of Theorem \ref{thm:2}.  Indeed, suppose there is a  $g'\in \Cmu$  with $g'\neq g^n \circ l$ for any $n\in \Z$ and $l\in \Cmu$ with $h_\mu(l) = 0$.  We have that $h_\mu(g^n \circ l) = |n|\lambda$, hence 
there is a $k\in \N$ such that  $h_\mu (g^k) < h_\mu (g')<h_\mu (g^{k+1})  $ whence we obtain either 
\[0<h_\mu (g^{-k}\circ (g')\inv) < h_\mu (g) \]
or 
\[0<h_\mu (g^{-k}\circ g') < h_\mu (g) \]
from \eqref{eq:entadd}.  This contradiction completes the proof of statement (2) of Theorem \ref{thm:2}.  \qed

%%%%%%%%%%%%%%%%%%%%%%%%%%%%%%%%%%%%%%%%%%

%%%%%%%%%%%%%%%%%%%%%%%%%%%%%%%%%%%%%%%%%%
\section{Proof of Theorems \ref{thm:1} and  \ref{thm:1'}}
We begin with a claim that reduces Theorems \ref{thm:1} and  \ref{thm:1'} to the case of affine transformations.
Recall that we identify the torus $\T^n$ with the quotient group $\R^n/\Z^n$.  We write $[x]$ for the equivalence class of $x$ in $\T^n$.   For $B\in \GL(n,\Z)$ we write $L_B$ for the induced map on $\T^n$ and for $v\in \R^n$ we write $T(v)$ for the toral rotation $[x]\mapsto [x+v]$.  

\def\E{\mathcal E}
By a \emph{$k$-dimensional linear foliation} $\E$ of the torus, we mean the partition of $\T^n$ by cosets of $H$, where  $H$ is a connected $k$-dimensional subgroup of $\T^n$. We say a linear foliation is \emph{irrational} if each leaf $\E([x])$ is dense in $\T^n$ and is the injective image of $\R^k$.

\begin{claim}\label{prop:affine}
Let $\E_1$ and $\E_2$ be $k_1$- and $k_2$-dimensional, irrational linear foliations of $\T^n$ with $1\le k_i$, $k_1 + k_2 = n$ and such that $\E_1([0]) \cap \E_2([0])$ contains no 1-dimensional subgroups.   Let $g \colon \T^n \to \T^n$ be a homeomorphism preserving the foliations $\E_j$.   Then $g$ is affine; that is, there are $B\in \GL(n, \Z)$ and $ v\in \R^n$ such that \[g= T(v)\circ L_ B.\]  

\end{claim}
\begin{proof}
Let $\td g$ be any lift of $g$ to $\R^n$, let $v = \td g(0)$ and set $\bar g\colon x \mapsto \td g(x) - v$.  Then $\restrict {\bar g }{\Z^n}$ is a homomorphism.  Write $\td \E_j$ for the lifts of the foliations to $\R^n$.  We note that $\td \E_1(x) \cap \td \E_2(y)$ contains exactly one point for each $x,y\in \R^n$ and the set $$\Xi:= \{ \td \E_1(n) \cap \td \E_2(m) \in \R^n \mid n,m\in \Z^n\}$$ is dense in $\R^n$. We check that $\Xi$ is closed under addition in $\R^n$ and that $\bar g(x+y) = \bar g(x)+\bar g (y) $ for $x,y \in \Xi$.  By the continuity of $\bar g$ we have that $\bar g$ is linear, and the claim holds.  
\end{proof}

\newcommand{\V}{\mathcal{V}}
\newcommand{\Id}{\mathrm{I}}
\begin{proof}[Proof of Theorems \ref{thm:1} and \ref{thm:1'}]
We prove both theorems simultaneously.  By Proposition \ref{prop:2} and Proposition \ref{prop:1}, respectively, for any $r\ge 1$, any $g\in \Cmu$ satisfying the hypotheses of  Theorem \ref{thm:1} and any $g\in \Cmu[\{\mu, \nu\}]$ satisfying the hypotheses of  Theorem \ref{thm:1'} preserves both foliations  $\Fols$ and $\Folu$.  Write \[\Gamma = \begin{cases}
\Cmu & \text{in Theorem \ref{thm:1}},\\
 \Cmu [\{\mu, \nu\}]  & \text{in Theorem \ref{thm:1'}}.\end{cases}\]

Recall that we have $A\in \GL(2,\Z)$ and $h\colon \T^2\to \T^2$ such that $L_A \circ h = h \circ  a$.  For any $g\in \Gamma$ we have that $h \circ g \circ h\inv $ preserves the linear stable and unstable  foliations induced by the dynamics of $L_A$.  These foliations satisfy the hypothesis of Claim \ref{prop:affine}, whence we conclude that $h \circ g \circ h\inv = T(v)\circ L_B$ for some $B\in \GL(2, \Z)$ and $v\in \R^2$.  We note that $L_B$ preserves the unstable foliation of $\T^2$ induced by the dynamics of $L_A$.  By the density of leaves of the (1-dimensional) unstable foliation for $L_A$, we conclude that  $L_A $ and $L_B$, hence $A$ and $B$, commute.  

Note that in the case of Theorem \ref{thm:1'}, one of $\mu$ or $\nu$ is not the measure of maximal entropy; we assume that $\mu$ is this measure. 
In the case of either theorem write \[H:= \{[v]\in \T^2\mid T(v) _*(h_*(\mu) ) = h_*(\mu)\} .\] %\{[v]\in \T^2 \mid h\inv \circ T(v)\circ h \in \Gamma\}.\] 
\begin{claim}
$H$ is finite.  
\end{claim}
\begin{proof}
 Recall that  $B\in \GL(n,\Z)$ is  said to be \emph{irreducible} if the only $L_B$-invariant proper closed subgroups of $\T^n$ are finite.  
%We recall that given a compact metric space $X$ and a Borel probability measure $\nu$, the set of $\nu$-preserving homeomorphism of $X$ is closed in the metric of uniform convergence.  
%Since $h_*(\mu) $ is non-atomic 
We verify that $H$ is a closed $L_A$-invariant subgroup of $\T^2$.   Since $A$ is irreducible, if $H$ were infinite we would have $H = \T^2$ which would imply $ h_*(\mu)$ is the Haar measure on $\T^2$.  But it is well known that the only $a$-invariant measure $\mu$ for which  $h_*(\mu)$ is the Haar measure is the measure of maximal entropy.  
Thus $H$ is finite.  %$H'$, and consequently $H$, are finite.  
\end{proof}
Writing $C(A)$ for the centralizer of $A$ in $\Gl(2,\Z)$, we know again  that  $C(A)$ is of the form \[  \{ \pm M^n \mid n\in \Z\}\]  for some hyperbolic matrix $M$.  
%Replacing $M$ with $-M$ if necessary 
Replacing $M$ with  $-M$ if needed we may find a $k$ such that 
 \[h \circ a \circ h \inv=L_A = L_{  M^{k}}.\]
Then for any $g\in \Gamma$ we may find $v\in \R^2$, $l \in \Z$, and $\sigma\in \{-1,1\}$ so that \begin{align}\label{eq:loo}
h \circ g \circ h\inv = T(v) \circ L_B = T(v)\circ L_{\sigma M^{l}}  .\end{align}
We calculate that for $g$ as above
\[g\circ a \circ g\inv \circ a^{-1} = h\inv \circ  T(v - M^{k}v  ) \circ h \]
whence we have that $v\in (L_{I-  M^{k}} )\inv H$ where $L_{I-  M^{k}}$ denotes the toral endomorphism induced by $I-  M^{k}\colon \R^2 \to \R^2$.  
In particular, $v$ has rational coordinates.  

Now, if for every $g\in \Gamma$, the corresponding $l$ in \eqref{eq:loo} is zero, it follows that the group $\Gamma$ is finite and the conclusion of each theorem follows with $m=0$.

We thus assume the existence of a $g\in \Gamma$ with infinite order and derive the remainder of the result.    
Fix an infinite order $g\in \Gamma$ and  corresponding $B, v, \sigma$ and $l\neq 0$ as in \eqref{eq:loo}.  We have that the orbit of $[0]$ under the map $T(v) \circ L_B$ is finite.  Indeed $T(v) \circ L_B$ has a unique fixed point, whence, by a change of coordinates on $\T^2$, we may interpret $T(v) \circ L_B$ as an automorphism and $[0]$ as a point with rational coordinates.   We thus find a $j$ so that $(T(v) \circ L_B)^j([0]) = [0]$.  We check from \eqref{eq:loo} that $$h\circ g^{m} \circ h\inv (x) = L_{B^{m}}(x) + (T(v) \circ L_B)^{m}([0])$$ for $x\in \T^2$, whence
\[h\circ g^{2jk} \circ h\inv (x) = L_{M^{2ljk}}(x) + (T(v) \circ L_B)^{2jk}([0]) = h\circ a^{2jl} \circ h\inv (x).\]
In particular, setting $m = 2jl$ we have $a^m\in \Gamma$.  Note this follows even in the case $r>\theta$.  % though we may have $l=m = 0$.  

 Write \[\Gamma' := \{h \circ \gamma \circ h\inv\mid \gamma \in \Gamma\}\] and
\[G = C(A) \ltimes (L_{I-  M^{k}} )\inv H\]
with multiplication \[(B, [v]) \cdot (B', [v']) = (BB', [Bv' + v]).\]

We abuse notation and identify $T(v)\circ L_B\in \Gamma'$ with $(B, [v]) \in G$ whence we obtain  a natural inclusion of subgroups 
\[\<M^{mk}\>\subset \Gamma'\subset G.\]
Since $\det (I- M^{k}) \neq 0$ and $H$ is finite, $G$ contains $\<M^{mk}\>$ as  a finite index subgroup.  Consequently, $\Gamma'$  contains $\<M^{mk}\>$ as  a finite index subgroup and the conclusion follows.  
\end{proof}

\bibliographystyle{AWBmath.num}
\bibliography{measRig}
\end{document}